\documentclass[12pt,twoside,a4paper]{article}

\usepackage{amsmath}
\usepackage{a4wide}
\usepackage{graphicx}
\usepackage{dsfont}
\usepackage{amsthm}
\usepackage{amssymb}
\usepackage{tikz}
\usetikzlibrary{matrix}

\newtheorem{theorem}{Theorem}
\newtheorem{definition}{Definition}
\newtheorem{example}{Example}
\newtheorem{lemma}{Lemma}
\newtheorem{remark}{Remark}
\newtheorem{corollary}{Corollary}
\newtheorem{proposition}{Proposition}
\newtheorem{assumption}{Assumption}

\title{Leibniz-Dirac structures and nonconservative systems
with constraints}
\author{\"{U}nver \c{C}ift\c{c}i\thanks{\small{uciftci@nku.edu.tr}} 
 \\
\small{ Department of Mathematics}, \small{Nam\i k Kemal University,}\\ \small{59030 Tekirda\u{g}, Turkey}\\[1.5ex]
\small{ Johann Bernoulli Institute for Mathematics and Computer Science,}\\ \small{University of Groningen,} \small{PO Box 407,}\\ 
\small{9700 AK Groningen, The Netherlands} \\
}

\begin{document}
\maketitle

\abstract{
Although conservative
Hamiltonian systems with
constraints can be formulated in
terms of Dirac structures, 
a more general framework is
necessary to cover
also dissipative systems such as
gradient and metriplectic systems
with constraints.
We define Leibniz-Dirac structures which
lead to a natural generalization of Dirac
and Riemannian structures, for instance. From modeling point of
view, Leibniz-Dirac structures make it easy to formulate 
implicit  dissipative Hamiltonian systems. 
We give their exact characterization in terms of
bundle maps from the tangent bundle to the
cotangent bundle and vice verse. Physical systems which can be formulated
in terms of Leibniz-Dirac structures are discussed.}


\vspace*{1cm}

\noindent 

\noindent
\noindent

\section{Introduction}
\label{}
Dirac structures embody a number of geometric
structures such as symplectic, Poisson, foliation,
complex geometries \cite{Courant90,Gualteri11}. Since their first introduction 
there have been a great number of work done over the years,
which is still growing.
One of the most striking features of Dirac
structures is that they can give a geometric
picture of Hamiltonian systems with constraints,
holonomic or nonholonomic \cite{JotzRatiu08}.
Nevertheless, Dirac structures are
insufficient in formulating non-conservative Hamiltonian
systems such as gradient systems and systems with
damping systems.
In that sense, recently some attempts have been done
to put these systems into a rather Hamiltonian form.

For example in \cite{Blankenstein03}, a generalization of
Dirac structures is given in terms of an
inner product of split sign on the Pontryagin bundle
instead of the natural symmetric pairing. We
specify this definition in order to cover the
physical examples which are aimed to be
put into the Hamiltonian context. In \cite{OrtegaPlanas-Bielsa04},
the authors use the notion of Leibniz structures
\cite{GrabowskiUrbanski99} which is a generalization
of Poisson structures, whose tensor is not necessarily
skew-symmetric. Our approach is quite similar but we 
work on the Pontryagin bundle and we
also deal with systems with constraints
on the manifold. 
In \cite{NguyenTurski01}, dissipative Hamiltonian
systems with constraints are studied
with Dirac's original method of
reduced brackets. 
For other recent work on the
generalizations of the conservative Hamiltonian systems
we refer to \cite{Balseiroetal09,CendraGrillo06}
and the references therein.
Our motivation in this paper is to 
give a generalization of Dirac structures and
study their geometric features in order to construct a general framework of
non-conservative Hamiltonian systems.

We define Leibniz-Dirac structures 
by weakening the defining properties of Dirac structures as follows:
Let $V$ be a vector space with its dual
denoted by $V^*$. A subspace
$L\subset V \oplus V^*$
is called a Dirac structure if it is maximally isotropic
under the symmetric pairing
\begin{equation*}
\langle (v_1,\eta_1), (v_2,\eta_2) \rangle_{+}=\frac{1}{2}(
\langle \eta_1 | v_2 \rangle + \langle \eta_2 | v_1 \rangle)
\end{equation*}
for all $(v_1,\eta_1),(v_2,\eta_2)\in V \oplus V^*$, where
$\langle \, | \, \rangle$ denotes
the natural pairing between vectors and
co-vectors. As a result, it is shown in \cite{Courant90}
for a Dirac structure that the equations,
which are called the \emph{characteristic equations},
\begin{equation}\label{crac intro}
\rho(L)^{\circ}=L \cap V^* \,\,\\\ \mbox{and} \,\, \\\
\rho^*(L)^{\circ}=L \cap V 
\end{equation}
are satisfied, where $\rho$ and $\rho^*$
denote the projections from
$V \oplus V^*$ onto the first and second factor respectively, and
($^{\circ}$) stands for the annihilation operator.
Accordingly, there exist skew-symmetric linear maps
$\Omega:\rho(L) \rightarrow \rho(L)^*$
and $\Pi:\rho^*(L) \rightarrow  \rho^*(L)^*$.
A Leibniz-Dirac structure (LD structure in short)
is defined to be a subspace $L \in V \oplus V^*$
such that either of the characteristic equations 
(\ref{crac intro}) is satisfied. 
Then it turns out that this definition is
can be equivalent to the existence of a subspace $E \in V$
(or $F \in V^*$) and a linear map
$\Omega:E\rightarrow E^*$
(or $\Pi:F \rightarrow  F^*)$.  

The definition of LD
structures is so broad that many geometric 
structures can be treated as LD structures.
Of course, Dirac structures form a subfamily
of LD structures but besides those; metric, metriplectic
and Leibniz structures are covered by the
LD structures. Having this in mind, developing the basic geometry of
LD structures further forms one of the ingredient of
the paper,
and it can be said that most of the 
results on LD structures derived 
in the present work are a refinement of
Dirac structures. We see that
LD structures share some properties of 
Dirac structures such as being a Lagrangian subspace
for a suitable inner product. Their 
extensions on manifolds is defined
the same way
as in the Dirac structures. Several properties of
linear and smooth LD structures are discussed.

Another ingredient of the present work
is to show that LD structures
are general enough to
express possibly dissipative implicit Hamiltonian systems
with constraints. Dynamics on smooth LD structures is
studied in detail and examples are presented 
to illustrate the theoretical part. Examples show that LD structures
are the proper geometric arena for numerous
physical systems.

The paper is organized as follows. In Section
\ref{smooth LD str} we define linear  LD structures on manifolds and
give their characterization in terms of linear maps in Theorem \ref{linear main}.
Then it is shown that LD structures are Lagrangian subspaces
with respect to suitable symmetric pairings. Smooth
LD structures on manifolds are defined in Section \ref{smooth LD str},
where we also
relate LD structures to Leibniz structures.  In Section \ref{Dynamics} we
study admissible functions on manifolds with LD structures, then
we study  Hamiltonian dynamics of LD manifolds. 
We present several physical examples which can also
be given in different formalisms.
The paper ends with some conclusions and
future questions.    

\section{Linear Leibniz-Dirac structures}
\label{linear LD str}

Let $V$ be an $n$-dimensional vector space and $V^*$ be
its dual space. Consider the direct product space
$V \oplus V^*$ and denote the projections from $V \oplus V^*$
onto $V$ and $V^*$ by $\rho$ and $\rho^*$, respectively.
If $L \in V \oplus V^*$ is a subspace, it is clear that 
$\mbox{ker}\rho|_L=L \cap V^*$ and $\mbox{ker}\rho^*|_L=L \cap V$ (cf. \cite{Courant90}).
Throughout $L \cap V$ (resp. $L \cap V^*$) will be regarded
either as a subspace of $V$ (resp. $V^*$) or $V \oplus V^*$.  
For a subspace $W\in V$ we denote the annihilator by $W ^{\circ}$. 
We denote by 
$\langle \eta | v \rangle$ the natural pairing of a co-vector
$\eta  \in V^*$ and a vector
$v \in V$. After the introduction of notational convention
we give the following definition.

\begin{definition}\label{linear definition}
A Leibniz-Dirac structure (LD structure for short) on $V$ is
a subspace $L \subset V \oplus V^*$ which satisfies 
at least one of the following conditions:
\begin{eqnarray}
\rho(L)^{\circ}&=&L \cap V^*  \label{characteristic 1} \\
\rho^*(L)^{\circ}&=&L \cap V.   \label{characteristic 2}
\end{eqnarray} 
LD structures satisfying (\ref{characteristic 1}) 
are called forward LD structures, accordingly the ones satisfying
(\ref{characteristic 2}) are called backward LD structures.
\end{definition}

The meaning of the adjectives `forward' and `backward'
will be understood after Theorem \ref{linear main}. But
before we give some immediate conclusions of the definition
of LD structures.   

Observe that if $L \in V \oplus V^*$ is a subspace, then we have
the following simple results from Linear Algebra: 
\begin{eqnarray}\label{dim}
\mbox{dim}(L \cap V)+\mbox{dim}(\rho^*(L))&=&
\mbox{dim}(L) \\
\mbox{dim}(L \cap V^*)+\mbox{dim}(\rho(L))&=&
\mbox{dim}(L).
\end{eqnarray}
So, one can conclude 
\begin{proposition}\label{items}
Let $L$ be a subspace of
$V \oplus V^*$ with $\mbox{dim}(V)=n$,
then the following are satisfied:

{\bf($i$)} If $L$ is a LD structure,
then $\mbox{dim}(L)=n$.

{\bf($ii$)} If $L$ is $n$-dimensional
and $\rho(L)^{\circ} \subset  L \cap V^* $,
then $L$ is a forward LD structure on $V$.

{\bf($iii$)} If $L$ is $n$-dimensional
and $\rho^*(L)^{\circ} \subset  L \cap V $,
then $L$ is a backward LD structure on $V$.
\end{proposition}

The equations
(\ref{characteristic 1}) and (\ref{characteristic 2}) are called the
\emph{characteristic equations} \cite{Courant90}.
Now the question arises: When does a subspace $L \in V \oplus V^*$
satisfies both of characteristic equations? The following 
proposition gives a partial answer.

First let us recall two bilinear pairings which are of
crucial significance in the theory of
Dirac structures \cite{Courant90}: 
\begin{eqnarray}
 \langle (v_1,\eta_1),(v_2,\eta_2) \rangle_{\mp}
=\frac{1}{2}\, \left(\langle \eta_1 | v_2 \rangle \mp \langle \eta_2 | v_1 \rangle \right)
\end{eqnarray}
for all $(v_1,\eta_1),(v_2,\eta_2)\in V \oplus V^*$.

\begin{proposition}\label{linear lagrangian}
 Let $L$ be a subspace of $V \oplus V^*$. If
$L$ is a Lagrangian space with respect to 
$\ll , \gg_+$ or $\ll , \gg_-$, then $L$ satisfies 
both of the characteristic equations.
\end{proposition}
\begin{proof}

If $L$ is Lagrangian with respect to $\ll , \gg_+$ or $\ll , \gg_-$
we have 
\begin{eqnarray*}
\langle \rho^*(L) \, | \, \rho(L \cap V) \rangle=
 \pm \, \langle \rho^*(L \cap V) \, | \, \rho(L) \rangle=0,
\end{eqnarray*}
and
\begin{eqnarray*}
\langle \rho^*(L \cap V^*) \, | \, \rho(L) \rangle=
 \pm \, \langle \rho^*(L) \, | \, \rho(L \cap V^*) \rangle=0.
\end{eqnarray*}
So, $L \cap V \subset  \rho^*(L)^{\circ}$ and
$L \cap V^* \subset  \rho(L)^{\circ}$
A dimension count gives the equalities, since $\mbox{dim}(L)=n$. 
\end{proof}

\begin{definition}\label{definition sym skew}
A LD structure is
called a Dirac structure or a symmetric Dirac structure
if it is a Lagrangian subspace with respect to 
$\ll , \gg_+$ or $\ll , \gg_-$, respectively. 
\end{definition}

We give a representation of LD structures,
which is an extension of a representation of Dirac structures \cite{Courant90}.

\begin{theorem}\label{linear ab}
Let $L$ be a LD structure on an $n$-dimensional vector space
$V$, then there exist two linear maps $A:\mathbb{R}^n \rightarrow V$ 
and $B:\mathbb{R}^n \rightarrow V^*$ such that
\begin{eqnarray}
\mbox{ker} \, A \cap \mbox{ker} \, B&=&\{0 \},  \label{eqab}  
\end{eqnarray}
and
\[
 \left(\mbox{Im} \, A\right)^\circ =
  B \left( \mbox{ker} \, A \right) \tag{\theequation a}  \label{eqab1} 
\]
if LD is a backward LD structure and
\[
 (\mbox{Im} \, B)^{\circ} =
  A \left( \mbox{ker} \, B \right) \tag{\theequation b}  \label{eqab2}
\]
otherwise.

Conversely, any structure $L$ on $V$ given by
\begin{equation}
L=\{ (A(y), B(y)); \, y \in \mathbb{R}^n \} \label{eqab3}
\subset V \oplus V^* 
\end{equation}
is a LD structure.
\end{theorem}
\begin{proof}
We only prove what is related to forward LD structures
and the other case is similar.
Let $L$ be a forward LD structure on $V$. If one chooses
a basis for $L$, then this is equivalent to giving two linear maps
$A:\mathbb{R}^n \rightarrow V$ 
and $B:\mathbb{R}^n \rightarrow V^*$ such that
the basis becomes
$(A(e_1),B(e_1)),...,(A(e_n),B(e_n))$,
where $e_1,\dots,e_n$ is the standard basis for $\mathbb{R}^n$.
Since $L$ is $n$-dimensional, (\ref{eqab}) is satisfied. 
Observe that $L \cap V^* =B \left(\mbox{ker} \, A\right)$ and
$\rho(L)=\mbox{Im} \, A $.
Then by the defining property
(\ref{characteristic 1}) of $L$, the relation (\ref{eqab1}) is satisfied.

Conversely, assume that $L$ is given by (\ref{eqab3}),
then by (\ref{eqab}) it is $n$-dimensional, and (\ref{eqab1}) implies
(\ref{characteristic 1}) which concludes the proof.
\end{proof}

Next we give another representation of
LD structures which gives
an equivalent picture of the notion of
LD structures. 

\begin{theorem}\label{linear main}
{\bf($i$)} A forward LD structure on $V$ 
can be given by a pair $(E,\Omega)$ where $E \subset V$ is a subspace
and $\Omega: E \rightarrow E^*$ is a linear map.

{\bf($ii$)} A backward LD structure on $V$ 
can be given by a pair $(F,\Pi)$ where $F \subset V^*$ is a subspace
and $\Pi: F \rightarrow F^*$ is a linear map.
\end{theorem} 

\begin{proof}

Only $(i)$ part of the Theorem will be proved,
the other part can be proved with a similar reasoning.

For a given pair $(E,\Omega)$ define $L\in V \oplus V^*$ by
\begin{equation}
L = \{ (v,\eta) ; \, v\in E, \, \eta - \Omega(v) \in E^{\circ} \}.
\end{equation}
It is clear that 
$\rho(L)=E$ and 
\begin{equation}
L \cap V^* = \{ \eta ; \, (0,\eta) \in L \}= \{ \eta ; \,  \eta - \Omega(0)=\eta \in E^{\circ}  \}=E^{\circ}.
\end{equation}
Then one concludes Equation \ref{characteristic 1} which
means that $L$ is a forward LD structure.

Conversely, for a given forward LD structure $L$ set $\rho(L)=E$.
Then a linear map $\Omega:E \rightarrow E^*$
can be defined for all $x \in L$ by $\Omega(\rho(x)):=\rho^*(x)|_E$. 
To show that it is well-defined,
consider vectors $x=(v,\eta),x'=(v,\eta') \in L$.
We need to show that $\eta|_E=\eta'|_E$. 
It is clear that $(0,\eta-\eta') \in L$ which implies that
$ \eta-\eta' \in L \cap V^*$. By the condition (\ref{characteristic 1}), this
is equivalent to saying that $(\eta-\eta')|_E=0$
or $\eta|_E=\eta'|_E$, as desired.
\end{proof}

Theorem \ref{linear main} makes clear where the naming `forward' and `backward'
LD structures come from.

Next we have a closer look at
the structures of the linear maps $\Omega$
and $\Pi$, so we will be
more clear about the motivation of
the definition of LD structures.
But first note that the kernel of 
$\Omega$ (resp. $\Pi$) is $L \cap V$
(resp. $L \cap V^*$). 

Let $\Omega^T:E \rightarrow E^*$ be the adjoint map of $\Omega$, i.e.
\begin{equation}
\langle \Omega^T(v_1) | v_2 \rangle := \langle \Omega^T(v_2) | v_1 \rangle
\end{equation}
for all $v_1,v_2 \in E$. Then one can define
a symmetric linear map $\Omega^+:E \rightarrow E^*$
and a skew-symmetric linear map $\Omega^-:E \rightarrow E^*$ by
\begin{eqnarray}
\langle \Omega^+(v_1)| v_2 \rangle:=  \frac{1}{2}\,
\left(\langle \Omega(v_1) | v_2 \rangle+ \langle 
\Omega^T(v_1) | v_2 \rangle \right),\\
\langle \Omega^-(v_1)| v_2 \rangle:=  \frac{1}{2}\,
\left(\langle \Omega(v_1) | v_2 \rangle
 - \langle \Omega^T(v_1) | v_2 \rangle \right),
\end{eqnarray}
respectively.
This allows the unique decomposition
\begin{equation}
 \Omega = \Omega ^+ +\Omega^-
\end{equation}
which will be of great importance in the sequel.
It is also possible to define the unique decomposition
of $\Pi$ into symmetric and skew-symmetric parts:
\begin{equation}
 \Pi = \Pi ^+ +\Pi^-.
\end{equation}

If $L$ is a LD structure then
\begin{eqnarray}
 \langle (v_1,\eta_1),(v_2,\eta_2) \rangle_{+}
=\langle \Omega^+(v_1) | v_2 \rangle 
\end{eqnarray}
and 
\begin{eqnarray}
 \langle (v_1,\eta_1),(v_2,\eta_2) \rangle_{-}
=\langle \Omega^-(v_1) | v_2 \rangle 
\end{eqnarray}
for all $(v_1,\eta_1),(v_2,\eta_2)\in L$. Therefore the following
is concluded.

\begin{corollary}
A LD structure $L$ is a Dirac structure (resp. symmetric Dirac structure)
if and only if the corresponding linear map $\Omega$ given in Theorem \ref{linear main} 
is purely skew-symmetric (resp. symmetric).
\end{corollary}

\begin{remark}
The converse of the result above is not generally true, that is,
(\ref{characteristic 1}) and (\ref{characteristic 2}) are not sufficient for $\Omega$ (or $\Pi$)
to be symmetric or skew-symmetric. For instance, if $\Omega$ is an
isomorphism between $V$ and $V^*$ then the characteristic equations are
satisfied. Because, in this case 
$\rho^*(L)=V^*$ and $L \cap V=\{0\}$.
\end{remark}

We can further conclude the following result.
It was originally given for Dirac structures in \cite{Courant90}, and
for LD structures the result was used in \cite{Blankenstein03} without proof. 

\begin{proposition}\label{prop_inner}
Let $L\in V \oplus V^*$ be a LD structure on $V$.

$(i)$ If $L$ is a forward LD structure then $L$ is
maximally isotropic with respect to some
inner product $\ll,\gg$ of split sign and of the form 
\begin{equation}\label{bilinear 1}
\ll  (v_1,\eta_1), (v_2,\eta_2)  \gg \, =  \langle  \eta_1 |  v_2 \rangle +
\langle  \eta_2 |  v_1  \rangle 
 - 2\, \Psi(v_1,v_2) , 
\end{equation}
for all $(v_1,\eta_1),(v_2,\eta_2) \in V \oplus V^*$,
where $\Psi$ is a symmetric bilinear form on $V$.

$(ii)$ If $L$ is a backward LD structure then $L$ is
maximally isotropic with respect to some
inner product $\ll,\gg$ of split sign and of the form
\begin{equation}\label{bilinear 2}
\ll  (v_1,\eta_1), (v_2,\eta_2)  \gg \, =  \langle  \eta_1 |  v_2 \rangle +
\langle  \eta_2 |  v_1  \rangle 
 - 2\, \Phi(\eta_1,\eta_2) , 
\end{equation}
for all $(v_1,\eta_1),(v_2,\eta_2) \in V \oplus V^*$,
where $\Phi$ is a symmetric bilinear form on $V^*$.
\end{proposition}
\begin{proof}
Only $(i)$ is proved as the proof of ($ii$) is completely analogous. 
We know by Theorem \ref{linear main} $(i)$
that $L$ corresponds to a pair
$(E,\Omega)$ where $E \in V$ is a subspace
and $\omega:E \rightarrow E^*$ is a linear map.
Observe that $\Omega$ can be extended to whole $V$ which is also
denoted by $\Omega$. 
This gives a symmetric bilinear form $\Psi$ on $V$ defined by
\begin{equation}
\Psi(v_1,v_2):= \frac{1}{2} \,
 \left( \langle \Omega(v_1) | v_2  \rangle +
   \langle \Omega(v_2)| v_1 \rangle \right).
\end{equation}
(We note here that
the extension of $\Omega$ is not unique so the
the symmetric bilinear form is not uniquely
defined, but this does not change the result.) 
Then it is straight forward to show that $L$ is isotropic
with respect to the symmetric bilinear form in (\ref{bilinear 1}).
It remains to show that (\ref{bilinear 1}) is an inner product of
split sign.   
After choosing a proper basis for $V \oplus V^*$, the result will be clear.
Let  $\alpha_1,...,\alpha_n$ be a basis of $V$ and 
$\beta_1,...,\beta_n$ be a basis of $V^*$ such that
$\langle  \beta_i \, | \,  \alpha_j  \rangle = \delta _i^j, \, i,j=1,...,n,$ where
$\delta$ is the Kronecker symbol.
As a basis of $V \oplus V^*$ one can choose $ (0,\beta_1),...,(0,\beta_n)
,(\alpha_1,\Omega(\alpha_1)),...,(\alpha_n,\Omega(\alpha_n)) ,$ then 
the matrix associated to the bilinear form in (\ref{bilinear 1}) becomes 
\[ \left( \begin{array}{cc}
{O}_n & I_n \\
I_n & O_n \\ \end{array} \right), \]
where $O_n$ is the $n \times n$ zero matrix and ${I}_n$
is the $n \times n$ identity matrix.
Accordingly, the basis given by
\begin{eqnarray*}
y_i&=&\frac{\sqrt{2}}{2}\, \left[ (0,\beta_i)+(\alpha_i,\Omega(\alpha_i)) \right], \\
x_i&=&\frac{\sqrt{2}}{2}\, \left[ (0,\beta_i)-(\alpha_i,\Omega(\alpha_i)) \right]
\end{eqnarray*}
gives the diagonal form
\[ \left( \begin{array}{cc}
 {I}_n  & O_n  \\ 
O_n & - {I}_n \\ \end{array} \right). \]
Then it is concluded that the bilinear form in (\ref{bilinear 1})
has signature $(n,n)$ with $n=\mbox{dim}(V)$. This concludes the proof.
\end{proof}

We can deduce from the proof of Proposition \ref{prop_inner}
that LD structures can be defined as deformations
of Dirac structures as follows.
Let $\mbox{Symm}(V)$ and $\mbox{Symm}(V^*)$ be the additive groups of
symmetric bilinear forms on $V$ and $V^*$ respectively,
and let $\mbox{Dir}(V)$, $\mbox{FLD}(V)$ and $\mbox{BLD}(V)$ denote
the spaces of Dirac structures, forward LD structures and
backward LD structures on $V$ respectively. Then we have

\begin{corollary}\label{deformation cor}
With the notation above,
one has the following inclusions:

 $(i)$ \begin{equation}
  \mbox{FLD}(V)  \hookrightarrow  \mbox{Symm}(V) \times \mbox{Dir}(V)
  \end{equation}

 $(ii)$ \begin{equation}
   \mbox{BLD}(V) \hookrightarrow \mbox{Symm}(V^*) \times \mbox{Dir}(V).
  \end{equation}
\end{corollary}
\begin{proof}
$(i)$ We show that $\mbox{FLD}(V)$ can be identified with
a subspace of $\mbox{Symm}(V) \times \mbox{Dir}(V)$. 
 Consider the map 
\begin{equation}
 \tau:\mbox{Symm}(V) \times \mbox{Dir}(V) \rightarrow \mbox{FLD}(V)
\end{equation}
defined by
\begin{equation}
 \tau(\psi,L)=\{(v,\eta+\psi(v)); \, (v,\eta) \in L\}.
\end{equation}
It can be shown that $\tau$ is surjective.
In fact, every forward LD structure $L$
has a representation $(E,\Omega)$
and $\Omega$ can be extended to $V$.
Further more $\Omega$
can be split into
$\Omega=\Omega^++\Omega^-$.
Then we have 
\begin{equation}
L=\{(v,\eta+\psi^+(v)); \, (v,\eta) \in L_1\},
\end{equation}
where $L_1$ is the Dirac structure given by $(E,\Omega^-)$.
  
Define a relation ``$\sim$`` on $\mbox{Symm}(V) \times \mbox{Dir}(V)$
by 
\begin{equation}
(\psi_1,L_1) \sim (\psi_2,L_2) \Leftrightarrow
 L_1=L_2 \, \mbox{and} \, 
\psi_1|_{\rho(L_1)}=\psi_1|_{\rho(L_1)}
\end{equation}
 which can be shown to be an equivalence relation.
Therefore we have the identification
 \begin{equation}
  \mbox{FLD}(V)  \approx  \mbox{Symm}(V) \times \mbox{Dir}(V) / \sim.
  \end{equation}

$(ii)$  Considering the map 
\begin{equation}
\nu:\mbox{Symm}(V^*) \times \mbox{Dir}(V) \rightarrow \mbox{BLD}(V)
\end{equation}
defined by
\begin{equation}
 \nu(\phi,L)=\{(v+\phi(\eta),\eta); \, (v,\eta) \in L\}
\end{equation}
gives the conclusion.
\end{proof}

The idea behind Corollary \ref{deformation cor} is
gauge equivalance of Dirac structures \cite{BursztynRadko03}
in which case the Dirac structures are deformed by skew-symmetric
bilinear maps.

Now we address
to the question: When a symmetric Dirac structure is also a Dirac structure?
But before we
recall the definition of a separable Dirac structures
\cite{vdSchaftMaschke12}, a notion which appears as a
generalization of Tellegen's theorem in circuit theory.
A Dirac structure $L \subset V \oplus V^*$
is a \emph{separable Dirac structure} if 
\begin{equation}\label{separable}
\langle \eta_1 | v_2 \rangle= 0 ,
\end{equation}
for all $(v_1,\eta_1), (v_2, \eta_2) \in L$.
It is ahon in \cite{vdSchaftMaschke12} that a subspace
$L \in V \oplus V^*$ is a seperable Dirac structure
if and only if
\begin{equation}
 L=K \oplus K^{\circ}
\end{equation}
for some subspace $K \in F$.

We then have
\begin{proposition}
A subspace $L  \subset V \oplus V^*$ is both a Dirac and a symmetric Dirac structure
if and only if  it is a separable Dirac structure.
\end{proposition}
\begin{proof}
$L$ is a Dirac and a symmetric Dirac structure, then
\begin{equation}\label{diraccon}
\langle \eta_1 | v_2 \rangle+\langle \eta_2 | v_1 \rangle=0 ,
\end{equation}
and
\begin{equation}\label{bdcon}
\langle \eta_1 | v_2 \rangle-\langle \eta_2 | v_1 \rangle=0
\end{equation}
for all $(v_1,\eta_1), (v_2, \eta_2) \in L$, respectively.
Then summing these equations gives Equation \ref{separable}.

Conversely, if $L$ is a separable Dirac structure,
it is easily seen by Equation (\ref{separable})
that the symmetric Dirac condition (\ref{bdcon}) is satisfied trivially.
\end{proof}

\section{Smooth Leibniz-Dirac structures}
\label{smooth LD str}

Let $M$ be a $n$-dimensional smooth manifold. Consider a smooth subbundle $L$ of the Pontryagin bundle
$TM \oplus T^*M$. We denote by the projections from $TM \oplus T^*M$ onto
$TM$ and $T^*M$ by $\rho$ and $\rho^*$, respectively.
Definition \ref{linear definition} can be given on a manifold as the following.
\begin{definition}
Let $L$ be a smooth vector subbundle
of $TM \oplus T^*M$. Then $L$  
 is called a Leibniz-Dirac structure (LD structure) if either of the following
equations holds:
\begin{eqnarray}
\rho(L)^{\circ}&=&L \cap T^*M \label{characteristic smooth 1}  \\
\rho^*(L)^{\circ}&=& L \cap TM . \label{characteristic smooth 2}
\end{eqnarray} 
LD structures satisfying (\ref{characteristic smooth 1}) 
are called forward LD structures, accordingly the ones satisfying
(\ref{characteristic smooth 2}) are called backward LD structures.
\end{definition}

\begin{remark}
 The equations 
(\ref{characteristic smooth 1}) and (\ref{characteristic smooth 2})
are not bundle equations, in general. 
However, (\ref{characteristic smooth 1}) implies 
\begin{equation}
 \rho(L) \subset(L \cap T^*M)^{\circ}
\end{equation}
and (\ref{characteristic smooth 2}) implies 
\begin{equation}
 \rho^*(L) \subset(L \cap TM)^{\circ}
\end{equation}
with the equality if the relations are bundle relations.
\end{remark}

By considering the preceding remark we
have the following which is  similar to the linear case.

\begin{proposition}\label{items smooth}
Let $L$ be a subbundle of
$TM \oplus T^*M$ with $\mbox{dim}(M)=n$,
then the following are satisfied:

{\bf($i$)} If $L$ is a LD structure,
then $\mbox{rank}(L)=n$.

{\bf($ii$)} If the rank of $L$ is equal to $n$
and $\rho(L)^{\circ} \subset  L \cap T^*M $ is
satisfied as a bundle equation,
then $L$ is a forward LD structure on $M$.

{\bf($iii$)} If the rank of $L$ is equal to $n$
and $\rho^*(L)^{\circ} \subset  L \cap TM $ is
satisfied as a bundle equation,
then $L$ is a backward LD structure on $M$.
\end{proposition}

\begin{proof}
 $(i)$ Since $\rho$ (resp. $\rho^*$) is a bundle map, there is an open dense
set on which $\rho(L)$ and hence $L\cap TM$ (resp. $\rho^*(L)$ and hence $L \cap T^*M$)
are bundles. 
Then the rank of $L$ on these points is $n$.
Since $L$ is a bundle one has that $\mbox{rank}(L(x))=n$ for all $x \in M$.  

$(ii)$ As $\rho(L)$ and $L\cap TM$
have constant rank by the hypothesis, we have the equation 
\begin{eqnarray}\label{dim smooth 1}
\mbox{dim}(L(x) \cap T_xM)+\mbox{dim}(\rho^*(L(x)))&=&
\mbox{dim}(L(x)) 
\end{eqnarray}
for all $x \in M$. Therefore $\rho(L(x))^{\circ} =  L(x) \cap T_x^*M $
for all $x \in M$, then one concludes that $\rho(L)^{\circ} =  L \cap T^*M $.

$(iii)$ As $\rho^*(L)$ and $L \cap T^*M$ have constant rank
by the hypothesis, we have the equation 
\begin{eqnarray}\label{dim smooth 2}
\mbox{dim}(L(x) \cap T_x^*M)+\mbox{dim}(\rho(L(x)))&=&
\mbox{dim}(L(x))
\end{eqnarray}
for all $x \in M$. Therefore the result follows.
\end{proof}

The set on which $\rho(L)$ and $L\cap TM$ (resp. $\rho^*(L)$ and $L \cap T^*M$)
are bundles is called the set of \emph{regular points} of $L$ \cite{Courant90}.

We proceed with the relation between
LD structures and Lagrangian subbundles of
$TM \oplus T^*M$.
The two bilinear pairings are defined by
\begin{eqnarray}
 \langle (v_1,\eta_1),(v_2,\eta_2) \rangle_{\mp}
=\frac{1}{2}\, \left(\langle \eta_1 | v_2 \rangle \mp \langle \eta_2 | v_1 \rangle \right)
\end{eqnarray}
for all $(v_1,\eta_1),(v_2,\eta_2)\in TM \oplus T^*M$. 

Proposition \ref{linear lagrangian} extends directly
to the following.

\begin{proposition}
 A subbundle $L \in TM \oplus T^*M$
satisfies both of the equations (\ref{characteristic smooth 1})
and (\ref{characteristic smooth 2})
if it is a Lagrangian subbundle with respect to
$\ll , \gg_+$ or $\ll , \gg _-$. 
\end{proposition}

\begin{definition}\label{definition sym skew smooth}
A subspace 
called a Dirac structure or a symmetric Dirac structure
if it is a Lagrangian subbundle with respect to
$\ll , \gg_+$ or $\ll , \gg _-$. 
\end{definition}

Now a locally defined representation of LD structures is given
as an extension of the linear case given in Theorem \ref{linear ab}. 

\begin{theorem}
Let $L$ be a LD structure on an $n$-dimensional manifold
$M$, then there exist two locally defined bundle maps $A:M \times \mathbb{R}^n \rightarrow TM$ 
and $B:M \times \mathbb{R}^n \rightarrow T^*M$ such that for all $m \in M$
\begin{eqnarray}
\mbox{ker} \, A_m \cap \mbox{ker} \, B_m &=&\{0 \},  \label{eqab smooth}  
\end{eqnarray}
and
\[
 \left(\mbox{Im} \, A_m\right)^\circ =
  B_m \left( \mbox{ker} \, A_m \right) \tag{\theequation a}  \label{eqab1 smooth} 
\]
if LD is a backward LD structure and
\[
 (\mbox{Im} \, B_m)^{\circ} =
  A_m \left( \mbox{ker} \, B_m \right) \tag{\theequation b}  \label{eqab2 smooth}
\]
otherwise. Here $A_m:\mathbb{R}^n \rightarrow T_mM$ and $B_m:\mathbb{R}^n \rightarrow T_m^*M$
are the linear maps defined at a fixed $m \in M$.

Conversely, any subbundle $L$ on $M$ given for all
$m \in M$ by
\begin{equation}
L(m)=\{ (A_m(y), B_m(y)); \, m \in M, \, y \in \mathbb{R}^n \} \label{eqab3 smooth}
\subset T_mM \oplus T_m^*M 
\end{equation}
is a LD structure.
\end{theorem}
\begin{proof}
Since $L$ is a subbundle, a choice of
a local basis of sections for $L$
gives two bundle maps $A:M \times \mathbb{R}^n \rightarrow TM$ 
and $B:M \times \mathbb{R}^n \rightarrow T^*M$.  
Then the remaining part of the proof is
obvious by the proof of Theorem \ref{linear ab}. 
\end{proof}

Another representation of LD structures is given as the following.

\begin{theorem}\label{smooth main}
{\bf($i$)} A forward LD structure $L$ on $M$ such that
$\rho(L)$ is a subbundle  
can be given by a pair $(E,\Omega)$, where $E \subset TM$ is a subbundle
and $\Omega: E \rightarrow E^*$ is a bundle map.

{\bf($ii$)} A backward LD structure $L$ on $M$ such that
$\rho^*(L)$ is a subbundle
can be given by a pair $(F,\Pi)$, where $F \subset T^*M$ is a subbundle
and $\Pi: F \rightarrow F^*$ is a bundle map.
\end{theorem} 

\begin{proof}
Only the proof of $(i)$ is given, a similar
reasoning holds for the case $(ii)$.

$(i)$ For a given pair $(E,\Omega)$ consider $L\subset TM \oplus T^*M$ defined by
\begin{equation}
L = \{ (v,\eta) ; \, v\in E, \, \eta - \Omega(v) \in E^{\circ} \}.
\end{equation}
Then $L$ is a subbundle since $E$ is a subbundle and
$\Omega$ is a bundle map.
It is also easy to see that 
$\rho(L)=E$ and 
\begin{equation}
L \cap T^*M = \{ \eta ; \, (0,\eta) \in L \}= \{ \eta ; \,  \eta - \Omega(0)=\eta \in E^{\circ}  \}=E^{\circ}.
\end{equation}
Thus Equation \ref{characteristic smooth 1} is
obtained.

Conversely, for a given forward LD structure $L$ set $\rho(L)=E$.
Then the map $\Omega:E \rightarrow E^*$
defined for all $x \in L$ by $\Omega(\rho(x)):=\rho^*(x)|_E$
is well-defined by the condition (\ref{characteristic smooth 1}).
Then it is a bundle map, since
$\rho^*$ is a bundle map.
\end{proof}

Having an equivalent picture of 
LD structures in terms of both 
subbundles of $TM \oplus T^*M$,
and pairs $(E, \Omega)$ and $(F,\Pi)$
is very useful as seen in the preceding section. 
To make use of this equivalent picture we make the following assumption.

\begin{assumption}\label{assumption}
 In the sequel, any the forward (resp. backward) LD structure $L$ will
be assumed to be given by a pair
$(E, \Omega)$ (resp. $(F,\Pi)$) in such a way that
the characteristic distribution $\rho(L)=E$ (resp. co-distribution $\rho^*(L)=F$) has constant rank.
\end{assumption}

For a LD structure, as in the linear case, one has the unique decomposition
\begin{equation}
 \Omega = \Omega ^+ +\Omega^-
\end{equation}
where $\Omega ^+$ is symmetric and $\Omega^-$ is
skew-symmetric.
Similarly one can define the unique decomposition
of $\Pi$ into symmetric and skew-symmetric parts:
\begin{equation}\label{splitting pi smooth}
 \Pi = \Pi ^+ +\Pi^-.
\end{equation}

\begin{proposition}\label{prop_inner smooth}
$(i)$ A forward LD structure $L$ is locally
maximally isotropic with respect to some
inner product $\ll,\gg$ of split sign and of the form 
\begin{equation}\label{bilinear smooth 1}
\ll  (v_1,\eta_1), (v_2,\eta_2)  \gg \, =  \langle  \eta_1 |  v_2 \rangle +
\langle  \eta_2 |  v_1  \rangle 
 - 2\, \Psi(v_1,v_2) , 
\end{equation}
for all $(v_1,\eta_1),(v_2,\eta_2) \in TM \oplus T^*M$,
where $\Psi$ is a symmetric covariant tensor field on $M$.

$(ii)$ A backward LD structure $L$ is
maximally isotropic with respect to some
inner product $\ll,\gg$ of split sign and of the form
\begin{equation}\label{bilinear smooth 2}
\ll  (v_1,\eta_1), (v_2,\eta_2)  \gg \, =  \langle  \eta_1 |  v_2 \rangle +
\langle  \eta_2 |  v_1  \rangle 
 - 2\, \Phi(\eta_1,\eta_2) , 
\end{equation}
for all $(v_1,\eta_1),(v_2,\eta_2) \in TM \oplus T^*M$,
where $\Phi$ is a symmetric contravariant tensor field on $M$.
\end{proposition}
\begin{proof}
The point here is that one can extend
$\Omega$ (resp. $\Pi$) to
$TM$ (resp. $T^*M$) locally
to define $\Psi$ (resp. $\Phi$).
The remainder of the proof is a straightforward extension of
Proposition \ref{prop_inner smooth} when
considered pointwise.     
\end{proof}

Dirac structures form a particular
subclass of LD structures, which include
symplectic, Poisson and foliation geometries.
Some other examples of LD structures
are discussed below. 

\begin{example}\label{ex grad}

Let $(M,g)$ be a pseudo-Riemannian manifold. 
The musical isomorphism $g^\sharp:T^*M \rightarrow TM$ of the
pseudo-Riemannian metric $g$
is a bundle map.
Then the graph of $g^\sharp$ given by
\begin{equation}
 L=\{(X,\eta); \, X=-g^\sharp(\eta) \} \in TM \oplus T^*M
\end{equation}
defines a LD structure on $M$ which is symmetric.
It will be explained in the next section that
this setting allows one to
study gradient control systems with constraints \cite{Blankenstein03}.
\end{example}

\begin{example}
A bundle map $\Pi : T^*M \rightarrow TM$
is called a \emph{Leibniz structure} \cite{OrtegaPlanas-Bielsa04}.
Then the graph of $\Pi$ is a LD structure
on $M$. These structures are shown to
model a very large family of physical
systems \cite{Morrison86,OrtegaPlanas-Bielsa04}. However,
LD structures also
allow to add some constraints when modeling
physical systems (cf. Section \ref{Dynamics}).
\end{example}

\section{Dynamics on LD manifolds}
\label{Dynamics}
Dynamic properties of LD structures are
given in this section. We first give the notion of
admissible functions. The main ingredient of 
this section is
a formulation of dissipative
Hamiltonian systems with constraints.

\subsection{Admissible functions}
\label{Admissible}
Admissible functions on LD manifolds are
defined as in the Dirac case \cite{Courant90}.
This definition makes sense for only 
backward LD structures
as being a generalization of the Poisson bracket. 

\begin{definition}
 Let $L$ be a backward LD structure
on a manifold $M$. A function $f$ on $M$ is called
an \emph{admissible function} if $df \in \rho^*(L)$.
\end{definition}
If $f$ is an admissible function, then 
$(X_f,df)\in L$ for some vector field
$X_f$ on $M$.

\begin{lemma}
 Let $L$ be a backward LD structure
on a manifold $M$. If $f$ and $g$ are
admissible functions then $fg$ is
also an admissible function.
\end{lemma}
\begin{proof}
By the hypothesis $(X_f,df),(X_g,dg) \in L$
for some vector fields $X_f$ and $X_g$ on $M$.
Then one computes
 \begin{eqnarray*}
g(X_f,df)+f(X_g,dg)
&=&(gX_f+fX_g,gdf+fdg) \\
&=&(gX_f+fX_g,d(fg)) \in L.
\end{eqnarray*}
So, $fg$ is an admissible function.
\end{proof}

Note that if $f$ and $g$ admissible functions,
then $fg$ is an admissible such that
$(X_{fg},d(fg) \in L$, where
$X_{fg}:=gX_f+fX_g$. 

In accordance with the Dirac case, a bracket $\{\{, \}\}$ on
admissible functions on $M$ can be defined by
\begin{equation}\label{defintion bracket}
\{\{f,g \}\}=X_f (g)=\langle dg | X_f \rangle
\end{equation}
for some $(X_f,df),(X_g,dg)\in L$.
If $(F,\Pi)$ is the corresponding backward LD pair
to $L$, then
\begin{equation}\label{bracked back}
\{\{f,g \}\}=\langle dg | \Pi(df) \rangle.
\end{equation}
Note that the bracket $\{\{, \}\}$
is well-defined as $\{\{f,g \}\}$ does not dependent on $X_f$ and $X_g$.

The following result is an extension of the Dirac case \cite{Courant90}. 

\begin{proposition}
With the notation above, the bracket $\{\{, \}\}$ on
admissible
functions satisfy the Leibniz identities:
\begin{eqnarray}
\{\{fg,h \}\}=f\{\{g,h \}\}+g\{\{f,h \}\}\\
\{\{h,fg \}\}=f\{\{h,g \}\}+g\{\{h,f \}\}
\end{eqnarray}
for all admissible functions $f,g,h$ on $M$.
\end{proposition}
\begin{proof}
Let $(X_f,df),(X_g,dg),(X_h,dh) \in L$. Then
\begin{eqnarray*}
\{\{fg,h \}\}&=&\langle dh | \Pi(d(fg)) \rangle \\
&=&\langle dh | \Pi(fdg+gdf) \rangle \\
&=&\langle dh | f\Pi(dg)+g \Pi (df) \rangle \\
&=&f \langle dh | \Pi(dg) \rangle +g \langle dh | \Pi (df) \rangle \\
&=&f\{\{g,h \}\}+g\{\{f,h \}\},
\end{eqnarray*}
since $\Pi$ is a bundle map,
and 
\begin{eqnarray*}
\{\{h,fg \}\}&=&\langle dh | \Pi(d(fg)) \rangle \\
&=&\langle d(fg) | \Pi(dh) \rangle \\
&=&\langle fdg+gdf | \Pi(dh) \rangle \\
&=&f \langle dg | \Pi(dh) \rangle +g \langle df | \Pi (dh) \rangle \\
&=&f\{\{h,g \}\}+g\{\{h,f \}\},
\end{eqnarray*}
$fg$ is an admissible function.
\end{proof}

Observe that the bracket $\{\{, \}\}$ splits into a skew-symmetric
bracket $\{,\}$ and a symmetric bracket $[,]$. In fact using the splitting 
(\ref{splitting pi smooth}) one has
\begin{eqnarray*}
\{\{f,g \}\}&=&\langle df | \Pi(dg) \rangle \\
&=&\langle df | \Pi^-(dg) \rangle + \langle df | \Pi^+(dg) \rangle \\
&=& \{f,g \}+[f,g],
\end{eqnarray*}
where
\begin{eqnarray}
\{f,g \}:=\langle df | \Pi^-(dg) \rangle
\end{eqnarray}
and
\begin{eqnarray}
 [f,g]:=\langle df | \Pi^+(dg) \rangle.
\end{eqnarray}

\begin{remark}
Note that the bracket of two admissible functions
is not again an admissible function, in general.
This is so even in the Dirac case,
however an integrability 
condition ensures the closedness of
the bracket \cite{Courant90}. 
\end{remark}

A (weak) integrability of LD structures on manifolds is
defined in accordance with the one on Dirac structures \cite{Courant90,DalsmovanderSchaft99}.

\begin{definition}
Let $L$ be a backward LD structure on a 
manifold $M$. If $\rho^*(L)^{\circ}=L \cap TM$
is involutive and the bracket $\{\{, \}\}$ 
is closed on admissible functions,
then $L$ is called a weakly integrable backward LD structure.
\end{definition}

Note that in the case of Dirac structures the integrability
is equivalent to the above conditions and additionally the Jacobi identity on
admissible functions \cite{DalsmovanderSchaft99}. (We already
assume that $\rho^*(L)$ has constant rank by Assumption \ref{assumption}.)

The following shows that the weak integrability
definition makes sense. 

\begin{proposition}\label{induced}
Let $L$ be a weakly integrable backward LD structure on a 
manifold $M$. If the foliation of $L \cap TM$ is denoted by
$\Phi$ and $M/\Phi$ is a manifold,  then $M/\Phi$ inherits a Leibniz structure.
\end{proposition}
\begin{proof}
 Functions on $M/\Phi$ can be
considered as $\Phi$-invariant functions on $M$. These functions
are the ones $f \in C^\infty (M)$ with 
$df(T\Phi)=0$. By the definition
these functions correspond to the admissible functions.
Therefore by the weak integrability assumption they are closed under
the bracket and give rise to an induced bracket on  $M/\Phi$,
which satisfies the Leibniz identities. 
\end{proof}

\subsection{Nonconservative systems with constraints}
\label{GIHS}

In this subsection we are concerned
with backward LD structures. This is because most of the
physical examples we study fall into that category.

Let $L$ be a LD structure on
a  manifold $M$, then one can extend the notion
of implicit Hamiltonian systems \cite{DalsmovanderSchaft99} to LD structures
as follows.  
\begin{definition}
Let $H:M \rightarrow \mathbb{R}$ be a Hamiltonian. The dissipative implicit
Hamiltonian system (DIHS) corresponding to
$(M,L,H)$ is given by 
\begin{equation}\label{ham}
(\dot{x},dH(x))\in L(x), \, \\\ x \in M.
\end{equation}
\end{definition}
In this setting, $\rho(L)$ describes the set of
admissible flows and $\rho^*(L)$ describes the set of algebraic constraints.
Assume that $L$ is represented by the pair $(F=\rho^*(L),\Pi)$, then
the DIHS corresponding to $(M,L,H)$
has a local representation
\begin{equation}\label{eqmot}
  \begin{split} 
  \dot{x}&=\Pi(x) \, \frac{\partial H}{\partial x}(x)+G(x)\, \lambda, \\  
 0&=G^{T}(x)\, \frac{\partial H}{\partial x}(x), 
  \end{split}
\end{equation}
where $\frac{\partial H}{\partial x}(x)$ stands for the column vector of
partial derivatives of $H$, and $G(x)$ is a full rank matrix
with $\mbox{Im}\,  G(x)=L(x) \cap T_xM$, and $\lambda$ are Lagrange multipliers
corresponding to the algebraic constraints
$0=G^{T}(x)\, \frac{\partial H}{\partial x}(x)$ \cite{DalsmovanderSchaft99}.

In terms of brackets one can obtain the equations of
motion as
\begin{equation}\label{bracket eq mot}
\frac{df}{dt}= \dot{x}(f)=\langle df | \dot{x} \rangle
= \langle df | \Pi(dH) \rangle=\{\{f,H \}\}
\end{equation}
for all admissible functions $f \in C^\infty(M)$.
Therefore, if the splitting (\ref{splitting pi smooth}) is considered, by (\ref{bracket eq mot})
one obtains
\begin{equation}\label{energy}
\frac{dH}{dt}=\langle \, dH  \, | \, \Pi^+(dH) \, \rangle=[H,H]. 
\end{equation}

If $L$ is a Dirac structure
then Equation \ref{energy} is nothing but
the conservation of energy.  
For nonconservative systems 
Equation \ref{energy}
has several meanings which will be cleared below.

\begin{example}[Gradient systems with constraints]
Let $(M,g)$ be a pseudo-Riemannian manifold and let
$F \subset T^*M$ be a subbundle. Consider the LD structure
\begin{equation}
 L=\{(X,\eta); \, X + g^\sharp(\eta) \in F^{\circ} \} \subset TM \oplus T^*M.
\end{equation}
Let $S: M \rightarrow \mathbb{R}$ be an \emph{entropy function} \cite{Morrison86}. 
Then the gradient system with constrained corresponding to
$(M,L,S)$ is defined by
\begin{equation}\label{grad def}
(\dot{x},dS(x))\in L(x), \, \\\ x \in M.
\end{equation}
Or, it can be represented by
\begin{equation}\label{eqmot grad}
  \begin{split} 
  \dot{x}&=-g^\sharp(x) \, \frac{\partial S}{\partial x}(x)+G(x)\, \lambda, \\  
 0&=G^{T}(x)\, \frac{\partial S}{\partial x}(x), 
  \end{split}
\end{equation}
where $G(x)$ is a full rank matrix
with $\mbox{Im}\, G(x)=L(x) \cap T_xM$ \cite{Blankenstein03}.

Then the equations of
motions in brackets read 
\begin{equation}\label{bracket eq mot grad}
\frac{df}{dt}=  -\langle df | g^\sharp(dS) \rangle=[f,S ]
\end{equation}
for all admissible functions $f \in C^\infty(M)$.
Recall here that the bracket $[,]$ is called the
\emph{Beltrami bracket} \cite{Crouch80}.
Eventually one obtains the equation
\begin{equation}\label{energy grad}
\frac{dS}{dt}=-\langle \, dS  \, | \, g^\sharp(dS) \, \rangle=[S,S] \le 0
\end{equation}
which is called the \emph{entropy equation} \cite{Blankenstein03}.

One of the examples of gradient systems with constraints
is RCL circuits with excess elements. This topic was
studied in \cite{Blankenstein05} by using LD
structures, but we believe that it is more
convenient to do a study particularly by
symmetric Dirac structures. 
\end{example}

\begin{example}[Metriplectic systems with constraints]
Let $M$ be a manifold and $F \in T^*M$ be a subbundle.
Let $P:T^*M \rightarrow TM$ be a Poisson structure and
$g$ be (possibly degenerate) Riemannian metric. Set a
Leibniz tensor by 
\begin{equation}
 \Pi=P-g^\sharp
\end{equation}
then the LD structure given by
\begin{equation}
 L=\{(X,\eta); \, X + \Pi(\eta) \in F^{\circ} \} \subset TM \oplus T^*M.
\end{equation}
is called a \emph{metriplectic structure} \cite{Morrison86}.
If $H:M \rightarrow \mathbb{R}$ is a smooth function,
then the system given by
\begin{equation}\label{ham met}
(\dot{x},dH(x))\in L(x), \, \\\ x \in M.
\end{equation}
is called a \emph{metriplectic system with constraints}.
The corresponding equations of
motions take the form 
\begin{equation}\label{bracket eq mot metrip}
\frac{df}{dt}=\langle df | P(dH) \rangle  -\langle df | g^\sharp(dH) \rangle=\{f,H\}+[f,H ]
\end{equation}
for all admissible functions $f \in C^\infty(M)$.
Then one obtains the equation
\begin{equation}\label{energy metrip}
\frac{dH}{dt}=-\langle \, dH  \, | \, g^\sharp(dH) \, \rangle=[H,H] \le 0
\end{equation}
which describes the dissipation of energy \cite{Blochetal96,NguyenTurski01}.
\end{example}
 
Now we discuss how to determine the $\lambda$ in Equation \ref{eqmot}, see
\cite{DalsmovanderSchaft99} for details.  
Assume that the $n \times k$ matrix $G(x)$ has rank $k \le n$. Then there
exists an $(n-k) \times n$ matrix $K(x)$ such that $K(x)G(x)=0.$
Therefore multiplying by $K(x)$ puts Equation \ref{eqmot} to the form
\begin{equation}\label{eqmot new}
  \begin{split} 
 \left[ \begin{array}{ccc}
K(x) \\
0 \end{array} \right]
  \dot{x}&= \left[ \begin{array}{ccc}
K(x)\Pi(x) \\
G(x)^T \end{array} \right]   \, \frac{\partial H}{\partial x}(x).
  \end{split}
\end{equation}  

\begin{example}[Mechanical systems with damping]
Let $Q$ be a manifold (configuration
space) and let $q=(q_1...,q_n)$ be a local coordinate system on $Q$.
Consider a Hamiltonian $H(q,p)$ on $M=T^*Q$ where
$(q,p)$ is the natural coordinate system on $T^*Q$. 
A \emph{mechanical system with damping} \cite{vdSchaft00} subject to 
$k$ independent kinematic constraints
\begin{equation}\label{constraints A}
 {A}^T(q) \, \dot{q}=0,
\end{equation}
can be defined by the representation
\begin{eqnarray*}
  \left[ \begin{array}{ccc}
\dot{q} \\
\dot{p} \end{array} \right] 
 &=&   \left[ \begin{array}{ccc}
O & I\\
-{I} & {R}(q) \end{array} \right]
  \left[ \begin{array}{ccc}
\frac{\partial H}{\partial q}(q,p) \\
\frac{\partial H}{\partial p}(q,p) \end{array} \right] 
+  \left[ \begin{array}{ccc}
{O} \\
{A}(q) \end{array} \right]\, \lambda, \\  
 O&=&   \left[ \begin{array}{ccc}
{O} & {A}^T(q) \end{array} \right] 
  \left[ \begin{array}{ccc}
\frac{\partial H}{\partial q}(q,p) \\
\frac{\partial H}{\partial p}(q,p) \end{array} \right],
 \end{eqnarray*}
where ${R}(q)$ is a semidefinite matrix.
Here the constraint forces $A(q)\lambda$ with
$\lambda \in \mathbb{R}^k$ are uniquely determined by the
requirement that the constraints (\ref{constraints A})
have to be satisfied for all time.  
Since $\mbox{rank}(A(q))=k$, one can find an
$(n-k)\times n$ matrix $K(q)$ of constant rank $n-k$
such that $K(q)A(q)=0.$ Then the above system
assumes the form
\begin{eqnarray*}
\left[ \begin{array}{ccc}
I & O\\
O & {K}(q) \\
O & O
 \end{array} \right]  \left[ \begin{array}{ccc}
\dot{q} \\
\dot{p} \end{array} \right] 
 &=&   \left[ \begin{array}{ccc}
O & I\\
-{K}(q) & K(q){R}(q) \\
O & A^T(q)
 \end{array} \right]
  \left[ \begin{array}{ccc}
\frac{\partial H}{\partial q}(q,p) \\
\frac{\partial H}{\partial p}(q,p) \end{array} \right] .
 \end{eqnarray*}  

In terms of LD structures one can define
the system in question as follows.  As 
the matrix $A(q)$ has rank $k$, its columns span
a co-distribution, say $G_0$, of constant rank on $Q$.
Set $G:=\pi^*(G_0)$ with
$\pi: T^*Q \rightarrow Q$ the natural projection. 
Let $B:T^*(T^*Q)\rightarrow T(T^*Q)$
be the canonical Poisson structure on $T^*Q$, which 
has the matrix form 
\begin{eqnarray*}
 \left[ \begin{array}{ccc}
O & I\\
-{I} & O \end{array} \right]
 \end{eqnarray*}
in the natural coordinates $(q,p)$.
Let $\tilde{R}:T^*(T^*Q)\rightarrow T(T^*Q)$ be the bundle map with
the matrix
\begin{eqnarray*}
 \left[ \begin{array}{cc}
O & O\\
O & -{R}(q) \end{array} \right]
 \end{eqnarray*}
and set the Leibniz structure
given by 
$\Pi:=B-\tilde{R}$.
Consider the LD structure
\begin{equation}
 L=\{(X,\eta); \, X-\Pi(\eta) \in G\} \subset TM \oplus T^*M,
\end{equation}
then the mechanical system with damping can be
defined by 
\begin{equation}
 (\dot{x},dH(x)),\, x=(q,p)\in M.
\end{equation}
The next example will illustrate
classical mechanical systems with damping more concretely.
\end{example}

\begin{example}(\cite{vdSchaft00,NguyenTurski01})
 We consider a particle moving in $\mathbb{R}^3$, subject to
the nonholonomic constraint $\dot{z}=y\dot{x}$ and a friction force proportional to
the particle velocity. The Hamiltonian is
given in terms of cartesian coordinates
$x,y,z$ and their conjugate momenta by 
\begin{equation}
 H(x,y,z,p_x,p_y.p_z)=\frac{1}{2}(p_x^2+p_y^2+p_z^2).
\end{equation}
The LD structure $L$ is given by
the characteristic distributions
\begin{eqnarray*}
 L \cap V &=& \mbox{span} \{\frac{\partial}{\partial p_z}-y \frac{\partial}{\partial p_x} \} \\
 \rho^*(L) &=& \mbox{span} \{ dx,dy,dz,y dp_z+dp_x,dp_y \}
\end{eqnarray*}
and the bundle map
\begin{equation*}
 \Pi=  B-\tilde{R},
\end{equation*}
where 
\begin{eqnarray*}
 B &=& \left[ \begin{array}{ccc}
O_3 & \ I_3\\
-{I}_3 & O_3 \end{array} \right], \\
\tilde{R}&=&\left[ \begin{array}{ccc}
{O}_3 & {O}_3\\
O_3 & -R  \end{array} \right],
\end{eqnarray*}
such that
\begin{equation*}
 R=\left[ \begin{array}{ccc}
\mu_1 & 0 & 0 \\
0 & \mu_2 & 0 \\
0 & 0 & \mu_3 \end{array} \right]
\end{equation*}
with $\mu_i(q)>0$ is the directional and space-dependent
damping coefficient \cite{NguyenTurski01}.
The equations of motion in brackets is given by
\begin{equation*}
 \dot{z}=\{\{z,H \}\}=\{z,H\}+[z,H]
\end{equation*}
where $z=(q,p)$, or more explicitly
\begin{equation}
 \{q_i,q_j\}=\{p_i,p_j\}=0, \quad \{q_i,p_j\}=\delta_i^j,
\end{equation}
and 
\begin{equation}
 [q_i,q_j]=[q_i,p_j]=0, \quad [p_i,p_j]=-\delta_i^j \mu_i,
\end{equation}
for all $i,j=1,2,3$. Therefore
the equations of motion read
\begin{eqnarray*}
  \left[ \begin{array}{ccc}
\dot{x} \\ 
\dot{y} \\
\dot{z} \\
\dot{p}_x \\
\dot{p}_y \\
\dot{p}_z \end{array} \right] 
 &=&  \left[ \begin{array}{cccccc}
0 & 0 & 0 & 1 & 0 & 0\\
0 & 0 & 0 & 0 & 1 & 0\\
0 & 0 & 0 & 0 & 0 & 1\\
-1 & 0 & 0 & \mu_1 & 0 & 0\\
0 & -1 & 0 & 0 & \mu_2 & 0\\
0 & 0 & -1 & 1 & 0 & \mu_3 \end{array} \right] 
  \left[ \begin{array}{ccc}
0 \\
0 \\
0 \\
p_x \\
p_y \\
p_z \\
 \end{array} \right] 
+  \left[ \begin{array}{ccc}
0 \\
0 \\
0 \\
y \\
0 \\
1 \\
 \end{array} \right]\, \lambda, \\  
 0&=&  yp_x-p_z.
 \end{eqnarray*}
Observe that not every point $x=(q,p) \in T^* \mathbb{R}^3$
satisfies  
\begin{equation}
 (\dot{x},dH(x)) \in L(x),
\end{equation}
but a proper subset 
\begin{equation}
 \chi_c=\{x\in T^* \mathbb{R}^3; \, dH(x) \in \rho^*(L(x))\}.
\end{equation}
It would be interesting to study reduction of the system to 
a subsystem on $\chi_c$ \cite{vdSchaft98}.
\end{example}

\section{Conclusions}

We have defined
linear and smooth Leibniz-Dirac structures which
are generalizations of
Dirac structures. We have studied
the geometry and dynamics of LD
structures both on linear spaces and manifolds.
It has been explained with
several examples
that LD structures are capable of formulating dissipative
implicit  Hamiltonian systems with constraints. 
We hope that LD structures will find applications in
physics and related areas.

However, there remain many questions
on both geometric and dynamics properties of
LD structures
unexplored, some of which are addressed herewith.
As it is known, a more general setting of
Dirac structures on Courant
algebroids has a growing importance \cite{Liuetal97,
Bursztynetal07}. Accordingly, LD structures
may be extended to vector bundles such as algebroids
\cite{GrabowskiUrbanski97,Balseiroetal09}.
Another topic is an investigation of 
transformations that preserve LD structures. 
For this end, one can use the notions of pushed forward
and pull back maps in the sense of 
\cite{BursztynRadko03}. This leads hopefully to symmetry reduction
of LD structures under Lie groups. 

Among LD structures, symmetric Dirac structures
have the richest geometry after Dirac structures.  
We believe that symmetric Dirac structures
are powerful tools in studying 
the geometry of physical systems such as gradient systems with
constraints \cite{Blankenstein03}
and incompressible viscous fluids \cite{NguyenTurski01}.

\section{Acknowledgements}
We are grateful to the anonymous referees for their  careful reading
and their helpful comments and suggestions which improved the paper essentially.

\bibliographystyle{unsrt}
\bibliography{bib}

\end{document}